\documentclass[a4paper,reqno]{amsart}

\usepackage[utf8]{inputenc}
\usepackage{amssymb}
\usepackage{enumitem}
\usepackage{mathrsfs}
\usepackage{mathtools}
\usepackage{amsmath}
\usepackage [dvipsnames] { xcolor }
\usepackage{multirow}
\usepackage[colorlinks=true]{hyperref}

\hypersetup{linkcolor=Violet,urlcolor=NavyBlue,citecolor=NavyBlue}

\setlist[enumerate]{label=\emph{(\roman*)}}

\usepackage{environ}

\newtheorem{theorem}{Theorem}[section]

\newtheorem{lemma}[theorem]{Lemma}
\newtheorem{proposition}[theorem]{Proposition}

\theoremstyle{definition}
\newtheorem{definition}[theorem]{Definition}
\newtheorem{remark}[theorem]{Remark}
\numberwithin{equation}{section}
\newcommand{\R}{\mathbb{R}}

\newcommand{\dd}{{\rm d}}

\begin{document}

\title[2D mass-critical ZK equation]{Nonexistence of minimal mass blow-up solution for the 2D cubic Zakharov-Kuznetsov equation}

\author{Gong Chen}
\address{School of Mathematics, Georgia Institute of Technology, Atlanta, GA, 30332-0160, USA.}
\email{gc@math.gatech.edu}

\author{Yang Lan}
\address{Yau Mathematical Sciences Center, Tsinghua University, 100084 Beijing, P. R. China.}
\email{lanyang@mail.tsinghua.edu.cn}

\author{Xu Yuan}
\address{Academy of Mathematics and Systems Science, Chinese Academy of Sciences, Beijing 100190, P. R. China.}
\email{xu.yuan@amss.ac.cn}

\begin{abstract}
For the 2D cubic (mass-critical) Zakharov-Kuznetsov equation,
\begin{equation*}
\partial_t\phi+\partial_{x_1}(\Delta \phi+\phi^3)=0,\quad (t,x)\in [0,\infty)\times \mathbb{R}^{2},
\end{equation*}
we prove that there exist no finite/infinite time blow-up solution with minimal mass in the energy space. This nonexistence result is in contrast to the one obtained by Martel-Merle-Raphaël~\cite{MMR1} for the mass-critical generalized Korteweg-de Vries (gKdV) equation. The proof relies on a refined ODE argument related to the modulation theory and a modified energy-virial Lyapunov functional with a monotonicity property.
\end{abstract}

\maketitle

\section{Introduction}
\subsection{Main result}
In this note, we consider the minimal mass dynamics for the following 2D cubic Zakharov-Kuznetsov equation,
\begin{equation}\label{equ:CP}
\partial_t\phi+\partial_{x_1}(\Delta \phi+\phi^3)=0,\quad (t,x)\in [0,\infty)\times \mathbb{R}^{2}.
\end{equation}
Recall that, the Cauchy problem is locally well-posed in the energy space $H^{1}(\R^{2})$ (see~\cite{LinPas,RibaudVento} and reference therein). More precisely, for any initial data $\phi_{0}\in H^{1}(\R^{2})$, there exists a unique (in a certain sense) maximal solution $\phi(t)$ of~\eqref{equ:CP} in $C([0,T);H^{1}(\R^{2}))$ and 
\begin{equation}\label{equ:blowCauchy}
    T<\infty\quad \mbox{implies}\quad 
    \|\nabla \phi(t)\|_{L^{2}}\to \infty\ \ \mbox{as}\ \ 
    t\uparrow T.
\end{equation}
Recall also that, $H^{1}$ solutions satisfy the conservation of mass and energy:
    \begin{equation*}
M(\phi(t))=\int_{\R^{2}}|\phi(t,x)|^{2}\dd x\quad \mbox{and}\quad 
E(\phi(t))=\int_{\R^{2}}\left(\frac{1}2|\nabla \phi(t,x)|^{2}-\frac{1}{4}|\phi(t,x)|^{4}\right)\dd x.
\end{equation*}
The symmetry group for $H^{1}$ solution of~\eqref{equ:CP} is given by
\begin{equation*}
  \phi(t,x) \mapsto  \sigma_{0}\lambda_{0}\phi\left(\lambda_{0}^{3}(t-t_{0}),\lambda_{0}(x_{1}-x_{1,0}),\lambda_{0}(x_{2}-x_{2,0})\right),
\end{equation*}
where $\left(\sigma_{0},\lambda_{0},t_{0},x_{1,0},x_{2,0}\right)\in \left\{-1,1\right\}\times\R^{+}\times \R\times\R^{2}$. In particular, the scaling symmetry keeps the $L^{2}$-norm invariant and hence the problem is \emph{mass-critical}.

\smallskip
The family of \emph{soliton/traveling wave} solutions
\begin{equation*}
    \phi(t,x)=\lambda_{0}Q\left(\lambda_{0}(x_{1}-x_{1,0}-\lambda_{0}^{2}t),\lambda_{0}(x_{2}-x_{2,0})\right),
\end{equation*}
plays an important role in our analysis of the minimal mass blow-up dynamics. Here, we denote by $Q(x)=q(|x|)$ the ground state of~\eqref{equ:CP} where $q>0$ satisfies
\begin{equation*}
    -q''-\frac{q'}{r}+q-q^{3}=0,\quad q'(0)=0\quad \mbox{and}\quad 
    \lim_{r\to \infty}q(r)=0.
\end{equation*}
It is well-known and easily checked that, for any $n\in \mathbb{N}$,
\begin{equation*}
    \left|q^{(n)}(r)\right|\lesssim r^{-\frac{1}{2}}e^{-r},\quad \mbox{for}\ r>1.
\end{equation*}
From the variational argument in~\cite{Weinstein}, the ground state $Q$ is related to the best constant in the following  Gagliardo-Nirenberg inequality
\begin{equation*}
    \int_{\R^{2}}|f(x)|^{4}\dd x
    \le 2\left(\int_{\R^{2}}|\nabla f(x)|^{2}\dd x\right)
    \left(\frac{\int_{\R^{2}}|f(x)|^{2}\dd x}{\int_{\R^{2}}|Q(x)|^{2}\dd x}\right),\quad \mbox{for any}\ f\in H^{1}(\R^{2}).
\end{equation*}
Following the above inequality, the conservation of energy and the blow-up criterion~\eqref{equ:blowCauchy} imply that any initial data $\phi_{0}\in H^{1}(\R^{2})$ with subcritical mass, \emph{i.e.} satisfying $\|\phi_{0}\|_{L^{2}}<\|Q\|_{L^{2}}$, generates a \emph{global and bounded} solution in $H^{1}(\R^{2})$.

\smallskip
The study of singularity formation for small supercritical mass $H^{1}$ initial data 
\begin{equation}\label{equ:massabove}
    \|Q\|_{L^{2}}<\|\phi_{0}\|_{L^{2}}<\|Q\|_{L^{2}}+\delta^{*},\quad 0<\delta^{*}\ll 1,
\end{equation}
has been developed in Farah-Holmer-Roudenko-Yang~\cite{FHRY}. More precisely, they show that for any initial data $\phi_{0}\in H^{1}(\R^{2})$ with $E(\phi_{0})<0$ and~\eqref{equ:massabove}, the corresponding solution $\phi(t)$ blows up in finite/infinite forward time. To the best of our knowledge, this is the first result for the study of singularity formation for~\eqref{equ:CP}.

\smallskip
In this note, we focus on the critical mass problem $\|\phi_{0}\|_{L^{2}}=\|Q\|_{L^{2}}$ and 
show that there exists no minimal mass blow-up solution. In other words, we prove that the subcritical mass condition is not sharp for the global existence of~\eqref{equ:CP}.

\smallskip
To state the main result, we recall the definition of minimal mass blow-up solution.
\begin{definition}\label{def:minimalmass}
We say that $\phi(t)$ is a minimal mass blow-up solution for~\eqref{equ:CP} if $\|\phi_{0}\|_{L^{2}}=\|Q\|_{L^{2}}$ and there exists $0<T\le \infty$ such that $\phi(t)\in C([0,T);H^{1})$ with 
\begin{equation*}
    \|\nabla \phi(t)\|_{L^{2}}\to \infty\quad \mbox{as}\quad t\uparrow T.
\end{equation*}
In addition, we say that $T$ is the blow-up time of $\phi(t)$.
\end{definition}

We now state the following main result of this note.
\begin{theorem}\label{thm:main}
    There exists no minimal mass blow-up solution for~\eqref{equ:CP}.
\end{theorem}

\begin{remark}
It is worth mentioning that, in Theorem~\ref{thm:main},  we do not treat the case of global minimal mass solutions blowing up only along a sequence of time. Technically, we should first obtain the blow-up estimate for all time, and then we could establish the exponential decay on the right-hand side of $y_{1}$(see Section~\ref{SS:Mini}) which helps us enter the monotonicity regime to control the geometric parameters. This is similar to the case of the mass-critical gKdV equation (see~\cite[Page 1872]{MMR1}).
\end{remark}

\subsection {Previous results} 
Actually, the study of minimal mass blow-up solution for nonlinear dispersive equation has a long history, especially for nonlinear Schr\"odinger equation (NLS). Consider the following mass-critical NLS equation:
\begin{equation*}
    i\partial_{t}\phi +\Delta \phi +|\phi|^{\frac{4}{d}}\phi=0,\quad (t,x)\in [0,\infty)\times \R^{d}.
\end{equation*}
The \emph{explicit minimal mass blow-up solution} can be constructed by the pseudo-conformal symmetry of this equation. For any $t>0$, we denote 
\begin{equation*}
    S_{\rm{NLS}}(t,x)=\frac{1}{t^{\frac{d}{2}}}e^{-i\frac{|x|^{2}}{4t}-\frac{i}{t}}Q_{{\rm{NLS}}}\left(\frac{x}{t}\right),
\end{equation*}
where $Q_{\rm{NLS}}\in H^{1}(\R^{d})$ is the unique ground state of the mass-critical NLS equation. Then we have $S_{{\rm{NLS}}}(t)$ is a blow-up solution  with
\begin{equation*}
    \|S_{{\rm{NLS}}}(t)\|_{L^{2}}= \|Q_{{\rm{NLS}}}\|_{L^{2}}\quad \mbox{and}\quad  \|\nabla S_{{\rm{NLS}}}(t)\|_{L^{2}}\sim \frac{1}{t}\ \mbox{as}\ t\downarrow 0.
\end{equation*}
In particular, by the seminal work of Merle~\cite{Merlenls}, the blow-up solution $S_{{\rm{NLS}}}(t)$ is exactly the unique 
(up to the symmetries)
minimal mass blow-up solution to the mass-critical NLS equation. 
Recall that, for the 2D mass-critical Zakharov system,
\begin{equation*}
\left\{
\begin{aligned}
i\partial_{t}u&=-\Delta u+nu,\\
\frac{1}{c^{2}}\partial_{t}^{2}n&=\Delta n+\Delta |u|^{2},
\end{aligned}
\right.\quad (t,x)\in [0,\infty)\times \R^{2},
\end{equation*}
the pseudo-conformal symmetry does not exist, and there is no minimal mass blow-up solution (see Glangetas-Merle~\cite{Glangetas}). We refer to ~\cite{Banica, BW, Com, Krieger, Raphaelnls} for the constructions of minimal mass blow-up solutions for other NLS-type equations.
We also refer to Merle~\cite{Merlenonnls} for the nonexistence of minimal mass blow-up solution of a type of inhomogeneous mass-critical NLS equation.

\smallskip
We now recall the related results for the generalized Korteweg-de Vries equation (gKdV) which is the closest model related to~\eqref{equ:CP}. For the mass-critical gKdV equation
\begin{equation*}
    \partial_{t}\phi+\partial_{x}\left(\partial_{x}^{2}\phi+\phi^{5}\right)=0,\quad (t,x)\in [0,\infty)\times \R,
\end{equation*}
the \emph{existence and description} of the minimal mass blow-up solution were first studied by Martel-Merle-Rapha\"el~\cite{MMR1}, and then the description of such solution was sharpened by Combet-Martel~\cite{Comkdv}. More precisely, we denote by $Q_{\rm{KdV}}\in H^{1}(\R)$ the unique ground state of the mass-critical gKdV equation. Based on the work of~\cite{MMR1}, we know that there exists a unique solution $S_{{\rm{KdV}}}(t)$ on $(0,\infty)$ such that 
\begin{equation*}
    \|S_{{\rm{KdV}}}(t)\|_{L^{2}}= \|Q_{{\rm{KdV}}}\|_{L^{2}}\quad \mbox{and}\quad  \|\partial_{x} S_{{\rm{KdV}}}(t)\|_{L^{2}}\sim \frac{1}{t}\ \mbox{as}\ t\downarrow 0.
\end{equation*}
Then, the sharp asymptotics, both in the time and space variables, were derived in~\cite{Comkdv}, for any order derivative of $S_{\rm{KdV}}(t)$.
We refer to Martel-Merle~\cite{MMDUKE} for the nonexistence of minimal mass blow-up solution of mass-critical gKdV equation, assuming an $L^{2}$-decay on the right of the initial data. We also refer to~\cite{KK2,KSnls,MMR,MMR2, MDann} for related results of blow-up dynamics for mass-critical dispersive models.

\smallskip

Last, we briefly survey the literature related to the \emph{soliton dynamics} of Zakharov-Kuznetsov models. For the 2D quadratic Zakharov-Kuznetsov equation, the asymptotic stability of soliton and stability of multi-soliton was studied by C\^ote-Mun\~oz-Pilod-Simpson~\cite{CMPS} via energy-virial estimate. However, the numerical computation of the coercivity for a specific Schr\"odinger operator, which is related to such virial estimate, does not hold for the 3D case. Very recently, based on a regularized transformation and numerical computation, Farah-Holmer-Roudenko-Yang~\cite{FHRY1} deduced a new virial estimate and thus they could extend the asymptotic stability of soliton for the 3D quadratic Zakharov-Kuznetsov equation. It is worth mentioning here that, in Pilod-Valet~\cite{PV1}, the authors studied the dynamics of the collision of two solitary waves for the 2 and 3-dimensional Zakharov–Kuznetsov equation. On the study of blow-up dynamics for~\eqref{equ:CP}, we refer to our recent work~\cite{CLY} for the full description of the near soliton behavior.
We also refer to Bozgan-Ghoul-Masmoudi-Yang~\cite{BGM} for a similar result which was proved independently. More importantly, the two proofs use different energy-virial Lyapunov functionals. Finally, we also point out that in Trespalacios~\cite{JT}, assuming the solution blows up, the author derived a lower bound for the blow-up rate for the 2D cubic Zakharov-Kuznetsov equation using the local well-posedness theory and growth estimates.

\subsection{Strategy of the proof}
Let us give a brief insight into the proof of Theorem~\ref{thm:main}. Indeed, 
the proof is surprisingly simple and relies on an energy-virial estimate introduced in~\cite{CLY}. For the sake of contradiction, we first assume that a minimal mass blow-up solution $\phi(t)$ exists. Then, we consider the geometric decomposition on the minimal mass blow-up solution in the energy space $H^{1}(\R^{2})$:
\begin{equation*}
    \varepsilon(t,y)=\lambda(t)\phi\left(t,\lambda(t)y+x(t)\right)-Q_{b}(t)(y).
\end{equation*}
Here, $Q_{b}$ is close to $Q$ in energy space $H^{1}(\R^{2})$ for $b$ small enough.
We now change the time variable as follows:
\begin{equation*}
    s=\int_{0}^{t}\frac{\dd \sigma}{\lambda^{3}(\sigma)}
    \Longleftrightarrow
    \frac{\dd s}{\dd t}=\frac{1}{\lambda^{3}(t)}.
\end{equation*}
Heuristically, the geometric parameters $(\lambda,b)$ satisfy 
\begin{equation*}
    \frac{\lambda_{s}}{\lambda}=-b \ \ \mbox{and}\ \ 
    b_{s}+\theta b^{2}=0\Longrightarrow 
    \frac{\dd }{\dd s}\left(\frac{b}{\lambda^{\theta}}\right)=0.
\end{equation*}
Here, $\theta\approx 1.66$ is a constant related to the integration of $\Lambda Q$ (see Section~\ref{SS:Nota} for the definition of $\theta$).
This implies that the function $({b}/{\lambda^{\theta}})$ is a constant along the flow. Therefore, in the rigorous analysis, we could establish the lower and upper bounds of the function $({b}/{\lambda^{\theta}})$ via the monotonicity properties of $(\lambda,b,\varepsilon)$. On the other hand, from the expansions of mass and energy for the minimal mass blow-up solution $\phi(t)$, we could deduce the upper bound of the function $(b/\lambda^{2})$. Last, combining the above-mentioned estimates for $(b/\lambda^{\theta})$ and $(b/\lambda^{2})$ with $\lambda(t)\to 0$ as $t\uparrow T$, we reach a contradiction. We mention here that, due to the same reason, the general strategy for the construction of minimal mass blow-up solution does not work in our case now (see~\cite[Page 1877]{MMR1} for more details).

\subsection{Notation and conventions}\label{SS:Nota}
For any $\beta=(\beta_{1},\beta_{2})\in \mathbb{N}^{2}$, we denote 
\begin{equation*}
    |\beta|=|\beta_{1}|+|\beta_{2}|\quad \mbox{and}\quad 
    \partial_{y}^{\beta}=\frac{\partial^{|\beta|}}{\partial_{y_{1}}^{\beta_{1}}\partial_{y_{2}}^{\beta_{2}}}.
\end{equation*}

Denote by $\mathcal{Z}(\R^{d})$ by the set of functions $f\in C^{2}(\R^{d})$ such that, for any $n\in \mathbb{N}$ there exist $C_{n}>0$ and $r_{n}>0$ such that 
\begin{equation*}
    \sum_{|\beta|=n}\left|\partial_{y}^{\beta}f(y)\right|
    \le C_{n}(1+|y|)^{r_{n}}e^{-\frac{|y|}{2}},\quad \mbox{on}\ \R^{d}.
\end{equation*}
For any $f\in L^{2}(\R^{2})$ and $g\in L^{2}(\R^{2})$, we denote the $L^{2}$-scalar product by 
\begin{equation*}
    (f,g)=\int_{\R^{2}}f(x)g(x)\dd x.
\end{equation*}

Let $\chi:\R\to [0,1]$ be a $C^{\infty}$ nondecreasing function such that
\begin{equation*}
    \chi_{|(-\infty,-2)}\equiv 0\quad \mbox{and}\quad 
    \chi_{|(-1,\infty)}\equiv 1.
\end{equation*}

Let ${\textbf{e}}_{1}$ be the first vector of the canonical basis of $\R^{2}$, that is, ${\textbf{e}}_{1}=(1,0)$.

\smallskip
We introduce the generator of the scaling symmetry
\begin{equation*}
    \Lambda f=f+\nabla \cdot f,\quad \mbox{for any}\ f\in H^{1}(\R^{2}).
\end{equation*}
We also define the linearized operator $\mathcal{L}$ around the ground state by 
\begin{equation*}
    \mathcal{L}f=-\Delta f+f-3Q^{2}f,\quad \mbox{for any}\ f\in H^{1}(\R^{2}).
    \end{equation*}

Recall that, there exists $\mu_{1}>0$, such that for any function $f\in H^{1}(\R^{2})$, 
\begin{equation}\label{est:coer}
   \left( \mathcal{L}f,f\right)\ge \mu_{1}\|f\|_{H^{1}}^{2}-\frac{1}{\mu_{1}}\left[(f,Q^{3})^{2}+(f,\partial_{y_{1}}Q)^{2}+(f,\partial_{y_{1}}Q)^{2}\right].
\end{equation}

In addition, we define the Schr\"odinger operator $\mathcal{A}$ for the dual problem by 
\begin{equation*}
\begin{aligned}
    \mathcal{A}f
    &=-\frac{3}{2}\partial_{y_{1}}^{2}f-\frac{1}{2}\partial_{y_{2}}^{2}f-\frac{1}{2}\left(3Q^{2}+6y_{1}Q\partial_{y_{1}}Q\right)f\\
    &+\frac{3}{\|Q\|^{2}_{L^{2}}}\left[(f,Q^{2}\partial_{y_{1}}Q)y_{1}Q+
    (f,y_{1}Q)Q^{2}\partial_{y_{1}}Q\right].
    \end{aligned}
\end{equation*}
Recall that, from~\cite[Section 16]{FHRY}, there exists $\mu_{2}>0$ such that for any $f\in H^{1}(\R^{2})$,
\begin{equation}\label{est:coer2}
   \left( \mathcal{A}f,f\right)\ge \mu_{2}\|f\|_{H^{1}}^{2}
   -\frac{1}{\mu_{2}}\left[(f,Q)^{2}+(f,\partial_{y_{1}}Q)^{2}+(f,\partial_{y_{1}}Q)^{2}\right].
\end{equation}

Recall also that, we consider the following two smooth functions:
\begin{equation*}
    \widehat{F}(\xi)=\frac{1}{\sqrt{2\pi}}\int_{\R}F(y_{2})e^{-i\xi y_{2}}\dd y_{2}\quad \mbox{with}\quad F(y_{2})=\int_{\R}\Lambda Q(y_{1},y_{2})\dd y_{1}.
\end{equation*}
In addition, the constant $\theta$ is defined by 
\begin{equation*}
    \theta
    =2\bigg(\int_{\R}\frac{|\widehat{F}(\xi)|^{2}}{1+|\xi|^{2}}\dd \xi\bigg)\bigg/
    \bigg(\int_{\R}|\widehat{F}(\xi)|^{2}\dd \xi\bigg).
\end{equation*}
By an elementary numerical computation, we find $\theta\approx 1.66$ (see~\cite[Appendix A]{CLY}).

\smallskip
Throughout this note, for a given small constant $\delta$, we will denote by $\alpha(\delta)$ a generic small constant with $\alpha(\delta)\to 0$ as $\delta\to 0$.
\subsection*{Acknowledgments}
The author G.C. was partially supported by NSF grant DMS-2350301 and by Simons foundation MP-TSM-0000225.
\section{Preliminaries}
\subsection{The localized profile}
In this subsection, we recall the definition of the localized profile $Q_{b}$ and some useful estimates related to $Q_{b}$. We start with the definition of the non-localized profile $P$ which grows as $y_{1}\to -\infty$.

\begin{lemma}[\cite{CLY}]\label{le:nonlocal}
There exists a smooth function $P\in C^{\infty}(\R^{2})$ with   $\partial_{y_{1}}P\in \mathcal{Z}(\R^{2})$
such that the following estimates hold. 
\begin{enumerate}
    \item \emph{First-type estimate.} We have
\begin{equation*}
\partial_{y_{1}}\mathcal{L}P=\Lambda Q,\quad \lim_{y_{1}\to \infty}\partial_{y_{2}}^{n}P(y_{1},y_{2})=0,\ \mbox{for any}\  n\in \mathbb{N},
\end{equation*}
\begin{equation*}
|(P,\nabla Q)|=0\quad \mbox{and}\quad (P,Q)=\frac{1}{4}\int_{\mathbb{R}}|F(y_{2})|^{2}\dd y_2>0.
\end{equation*}
\item \emph{Second-type estimate.} For any $\beta=(\beta_{1},\beta_{2})\in \mathbb{N}^{2}$, there exists $C_{1\beta}>0$ such that 
\begin{equation*}
\begin{aligned}
\left|\partial_{y}^{\beta}P(y_{1},y_{2})\right|&\le C_{1\beta}e^{-\frac{|y_{2}|}{3}},\quad \mbox{on}\ \R^{2},\\
\left|\partial_{y}^{\beta}P(y_{1},y_{2})\right|&\le C_{1\beta}e^{-\frac{|y|}{3}},\quad \ \mbox{on}\ (0,\infty)\times \R.
\end{aligned}
\end{equation*}
In addition, for any $\beta=(\beta_{1},\beta_{2})\in \mathbb{N}^{2}$ with $\beta_{1}\ne 0$, there exists $C_{2\beta}>0$ such that 
\begin{equation*}
\left|\partial_{y}^{\beta}P(y_{1},y_{2})\right|\le C_{2\beta}e^{-\frac{|y|}{3}},\quad \ \mbox{on}\ \mathbb{R}^{2}.
\end{equation*}
\end{enumerate}
\end{lemma}

\begin{proof}
The proof of the existence for $P$ relies on the inversion of $\mathcal{L}$. We refer to~\cite[Proposition 2.1 and Lemma 2.2]{CLY} for the details of the proof.
\end{proof}
We now define the localized profile as follows
\begin{equation*}
\chi_{b}(y_{1})=\chi(|b|^{\frac{3}{4}}y_{1})\quad \mbox{and}\quad 
Q_{b}(y)=Q(y)+b\chi_{b}(y_{1})P(y).
\end{equation*}

In addition, we define the following error term related to the localized profile
\begin{equation*}
     \Psi_{b}=\partial_{y_{1}}\left(-\Delta Q_{b}+Q_{b}-Q_{b}^{3}\right)-b\Lambda Q_{b}.
     \end{equation*}

\begin{lemma}[\cite{CLY}]\label{le:local}
The localized profile $Q_{b}$ satisfies the following estimates.

\begin{enumerate}
    \item \emph{Pointwise estimates on $Q_{b}$.} For any $y\in\mathbb{R}^2$, we have
\begin{equation*}
|Q_{b}(y)|\lesssim e^{-\frac{|y|}{3}}
+|b|e^{-\frac{|y_{2}|}{3}}\mathbf{1}_{[-2,0]}(|b|^{\frac{3}{4}}y_1).
\end{equation*}
In addition, for any $y\in\mathbb{R}^2$ and $n\in\mathbb{N}^{+}$, we have 
\begin{equation*}
|\partial_{y_1}^n Q_b(y)|\lesssim e^{-\frac{|y|}{3}}+ |b|^{1+\frac{3n}{4}}e^{-\frac{|y_{2}|}{3}}\mathbf{1}_{[-2,-1]}(|b|^{\frac{3}{4}}y_1).
\end{equation*}

    \item \emph{Pointwise estimates on $\Psi_{b}$.}
    For any $y\in \R^{2}$, we have 
    \begin{equation*}
\begin{aligned}
\left|\Psi_{b}(y)\right|
&\lesssim |b|^{2}\left(e^{-\frac{|y|}{3}}+e^{-\frac{|y_{2}|}{3}}\mathbf{1}_{[-2,0]}(|b|^{\frac{3}{4}}y_1)\right)\\
&+|b|^{\frac{7}{4}}e^{-\frac{|y_{2}|}{3}}\mathbf{1}_{[-2,-1]}(|b|^{\frac{3}{4}}y_1).
\end{aligned}
\end{equation*}
In addition, for any $y\in \mathbb{R}^{2}$ and $n\in \mathbb{N}^{+}$, we have 
\begin{equation*}
|\partial_{y_1}^n \Psi_b(y)|\lesssim |b|^{2}e^{-\frac{|y|}{3}}+|b|^{1+\frac{3}{4}(n+1)}e^{-\frac{|y_{2}|}{3}}\mathbf{1}_{[-2,-1]}(|b|^{\frac{3}{4}}y_1).
\end{equation*}

    \item \emph{Mass and energy of $Q_{b}$.} We have 
    \begin{equation*}
    \begin{aligned}
    \left|E(Q_{b})+b(P,Q)\right|&\lesssim b^{2},\\
    \left|\int_{\R^{2}}Q_{b}^{2}\dd y-\int_{\R^{2}}Q^{2}-2b(P,Q)\right|&\lesssim |b|^{\frac{5}{4}}.
    \end{aligned}
\end{equation*}

\item \emph{$L^{2}$-scalar product with $Q$.} We have 
\begin{equation*}
    \left|(\Psi_b, Q)+\frac{b^2}{2}\int_{\mathbb{R}}\frac{|\widehat{F}(\xi)|^2}{1+|\xi|^2}\,\dd \xi\right|\lesssim |b|^3.
\end{equation*}
\end{enumerate}
    \end{lemma}

    \begin{proof}
        The proof relies on the definition of localized profile $Q_{b}$ and Lemma~\ref{le:nonlocal}. We refer to~\cite[Lemma 2.3]{CLY} for the details of the proof.
    \end{proof}

\subsection{Geometrical decomposition of the flow}\label{SS:Geome}
In this subsection, we recall the modulation theory for the solution of~\eqref{equ:CP} near the soliton manifold. We start with the following variation property of the ground state $Q$. The proof is similar to~\cite[Proposition 3.7]{Raphael} and we omit it.
\begin{lemma}[Variation property of $Q$]\label{le:varia}
There exists a constant $\delta_{1}>0$ such that the following hold. For any $0<\delta_{2}<\delta_{1}$ and $u\in H^{1}$ with 
\begin{equation*}
E(u)\le \delta_{2}\int_{\R^{2}}|\nabla u|^{2}\dd x\quad 
\mbox{and}\quad 
\int_{\R^{2}}u^{2}\dd x\le \int_{\R^{2}}Q^{2}\dd x+\delta_{2},
\end{equation*}
there exist $(\sigma_{0},\lambda_{0},x_{0})\in \left\{-1,1\right\}\times (0,\infty)\times \R^{2}$ such that 

\begin{equation*}
\left\|Q(x)-\sigma_{0}\lambda_{0}u(\lambda_{0} x+x_{0})\right\|_{H^{1}}\le \alpha(\delta_{2}).
\end{equation*}
\end{lemma} 

We assume that there exist $(\overline{\lambda}(t),\overline{x}(t))\in (0,\infty)\times \R^{2}$ and $\overline{\varepsilon}(t)\in H^{1}$ such that 
\begin{equation}\label{equ:assumdecom}
    \phi(t,x)=\frac{1}{\overline{\lambda}(t)}\left(Q+\overline{\varepsilon}\right)\left(t,\frac{x-\overline{x}(t)}{\overline{\lambda}(t)}\right),\quad \mbox{with}\ \|\overline{\varepsilon}(t)\|_{L^{2}}\le \kappa\le \kappa^{*}.
\end{equation}
Here, $0<\kappa^{*}\ll 1$ is a small enough universal constant.

\smallskip
First, we recall the following modulation theory for the solution of~\eqref{equ:CP}.
\begin{lemma}[\cite{CLY}]\label{le:decom}
Assuming~\eqref{equ:assumdecom} on $[0,t_{0}]$, then there exist $C^{1}$ functions $(\lambda,x,b): [0,t_{0}]\mapsto(0,\infty)\times \R^{2}\times \R$ such that 
\begin{equation}\label{equ:vare}
    \varepsilon(t,y)=\lambda(t)\phi(t,\lambda(t)y+x(t))-Q_{b(t)}(y),\quad \mbox{for any}\ t\in [0,t_{0}],
\end{equation}
satisfies the orthogonality conditions
\begin{equation}\label{equ:orthe}
\left(\varepsilon(t),Q\right)=\left(\varepsilon(t),Q^{3}\right)=\left|\left(\varepsilon(t),\nabla Q\right)\right|=0.
\end{equation}
Moreover, we have the following smallness estimates
    \begin{equation*}
        |b(t)|+\|\varepsilon(t)\|_{L^{2}}+\left|1-\frac{\lambda(t)}{\overline{\lambda}(t)}\right|\lesssim \alpha(\kappa)\quad 
        \mbox{and}\quad 
        \|\varepsilon(t)\|_{H^{1}}\lesssim \alpha\left(\|\overline{\varepsilon}(t)\|_{H^{1}}\right).
    \end{equation*}
\end{lemma}
\begin{proof}
    The proof of the decomposition proposition relies on the implicit function Theorem. We refer to~\cite[Propostion 3.1]{CLY} for the details of the proof.
\end{proof}

Second, we deduce the rough estimates related to the geometric parameters $(\lambda,b)$ for the minimal mass solution from the conservations of mass and energy.
\begin{lemma}\label{le:blambda}
Under the assumption of Lemma~\ref{le:decom}, and assume moreover $\|\phi_{0}\|_{L^{2}}=\|Q\|_{L^{2}}$. Then the following estimates hold.
\begin{enumerate}
    \item \emph{Estimate induced by the mass}. For any $t\in [0,t_{0}]$, we have 
    \begin{equation}\label{est:massrough}
        \|\varepsilon(t)\|^{2}_{L^{2}}\lesssim -b(t)\lesssim \|\varepsilon(t)\|^{2}_{L^{2}}.
    \end{equation}

    \item \emph{Estimate induced by the energy.} For any $t\in [0,t_{0}]$, we have 
    \begin{equation}\label{est:energyrough}
       0\le  E(\phi_{0})\lesssim \frac{|b(t)|}{\lambda^{2}(t)}+\frac{\|\varepsilon(t)\|^{2}_{H^{1}}}{\lambda^{2}(t)}\lesssim E(\phi_{0}).
        \end{equation}
\end{enumerate}
    \end{lemma}

    \begin{proof}
        Proof of (i). We claim that
        \begin{equation}\label{equ:e2b}
\int_{\R^{2}}\varepsilon^{2}\dd y+2b(P,Q)=\alpha(\kappa)\left(\|\varepsilon\|_{L^{2}}^{2}+|b|\right).
        \end{equation}
        
        Indeed, from \eqref{equ:vare}, \eqref{equ:orthe} and the conservation of mass, we deduce that 
        \begin{equation*}
        \begin{aligned}
            \int_{\R^{2}}\phi_{0}^{2}\dd y
            &=\int_{\R^{2}}Q^{2}\dd y+\int_{\R^{2}}\varepsilon^{2}\dd y+2b(P,Q)\\
            &+\int_{\R^{2}}Q_{b}^{2}\dd y-\int_{\R^{2}}Q^{2}\dd y-2b(P,Q)+2b(\varepsilon,\chi_{b}P).
            \end{aligned}
        \end{equation*}
        Using the definition of $P$ in Lemma~\ref{le:nonlocal} and the definition of $\chi_{b}$, 
        \begin{equation*}
            \|\chi_{b}P\|_{L^{2}}^{2}
            \lesssim\int_{\R}\int_{-2|b|^{-\frac{3}{4}}}^{0}
            e^{-\frac{2|y_{2}|}{3}}\dd y_{1}\dd y_{2}+\int_{\R}\int_{0}^{\infty}e^{-\frac{2|y|}{3}}\dd y_{1}\dd y_{2}\lesssim |b|^{-\frac{3}{4}}.
        \end{equation*}
        It follows from Lemma~\ref{le:local}, Lemma~\ref{le:decom} and the Cauchy-Schwarz inequality that 
        \begin{equation*}
           \left| \int_{\R^{2}}Q_{b}^{2}\dd y-\int_{\R^{2}}Q^{2}\dd y-2b(P,Q)+2b(\varepsilon,\chi_{b}P)\right|\lesssim \alpha(\kappa)\left(|b|+\|\varepsilon\|^{2}_{L^{2}}\right).
            \end{equation*}
            Combining the above identity and estimates with $\|\phi_{0}\|_{L^{2}}=\|Q\|_{L^{2}}$, we obtain~\eqref{equ:e2b} and thus the proof of~\eqref{est:massrough} is directly complete.
        
        \smallskip
        Proof of (ii). From \eqref{equ:vare}, \eqref{equ:orthe} and the conservation of energy, we deduce that     
        \begin{equation*}
        \begin{aligned}
    2\lambda^{2}E(\phi_{0})
    &=\left(\mathcal{L}\varepsilon,\varepsilon\right)
    +2\left(E(Q_{b})+b(P,Q)\right)-\left(\int_{\R^{2}}\varepsilon^{2}\dd y+2b(P,Q)\right)\\
    &-\frac{1}{2}\int_{\R^{2}}\left(\varepsilon^{4}+4Q_{b}\varepsilon^{3}+6Q_{b}^{2}\varepsilon^{2}-6Q^{2}\varepsilon^{2}\right)\dd y\\
    &+2\int_{\R^{2}}\varepsilon\left(-\Delta(Q_{b}-Q)-(Q_{b}^{3}-Q^{3})\right)\dd y.
    \end{aligned}
        \end{equation*}
        First, from Lemma~\ref{le:nonlocal} and~\eqref{equ:e2b}, we have 
        \begin{equation*}
\left|E(Q_{b})+b(P,Q)\right|+\left|\int_{\R^{2}}\varepsilon^{2}\dd y+2b(P,Q)\right|\lesssim \alpha(\kappa)\left(\|\varepsilon\|_{L^{2}}^{2}+|b|\right).
        \end{equation*}
        Second, using the Gagliardo-Nirenberg and Cauchy-Schwarz inequalities, 
        \begin{equation*}
        \begin{aligned}
&\left|\int_{\R^{2}}\left(\varepsilon^{4}+4Q_{b}\varepsilon^{3}+6Q_{b}^{2}\varepsilon^{2}-6Q^{2}\varepsilon^{2}\right)\dd y\right|\\
&\lesssim 
\|\varepsilon\|_{L^{4}}^{4}+
\|\varepsilon\|_{L^{4}}^{3}\|Q_{b}\|_{L^{4}}
+|b|\|\varepsilon\|_{L^{4}}^{2}
+|b|\|\varepsilon\|^{2}_{L^{2}}
\lesssim \alpha(\kappa)\left(\|\varepsilon\|_{H^{1}}^{2}+|b|\right).
\end{aligned}
\end{equation*}
Next, using the definition of $Q_{b}$ and $P$, we compute
\begin{equation*}
\begin{aligned}
    -\Delta\left(Q_{b}-Q\right)&=-b\chi_{b}\Delta P-2b\chi'_{b}\partial_{y_{1}}P-b\chi''_{b}P,\\
    -(Q_{b}^{3}-Q^{3})&=-3b\chi_{b}PQ^{2}-3b^{2}\chi_{b}^{2}P^{2}Q-b^{3}\chi_{b}^{3}P^{3}.
    \end{aligned}
\end{equation*}
Using again the definition of $P$ in Lemma~\ref{le:local} and the definition of $\chi_{b}$,
\begin{equation*}
\begin{aligned}
    \left\|\chi_{b}\Delta P\right\|_{L^{2}}^{2}+\|\chi'_{b}\partial_{y_{1}}P\|_{L^{2}}^{2}+\|\chi''_{b}P\|_{L^{2}}^{2}+\|\chi_{b}P\|_{L^{6}}^{6}&\lesssim |b|^{-\frac{3}{4}}.
    \end{aligned}
\end{equation*}
It follows from the Cauchy-Schwarz inequality that 
\begin{equation*}
    \left|\int_{\R^{2}}\varepsilon\left(-\Delta(Q_{b}-Q)-(Q_{b}^{3}-Q^{3})\right)\dd y\right|\lesssim \alpha(\kappa)\left(\|\varepsilon\|_{L^{2}}^{2}+|b|\right).
    \end{equation*}
    Based on the above identity an estimates, we obtain
    \begin{equation*}
2\lambda^{2}E(\phi_{0})=\left(\mathcal{L}\varepsilon,\varepsilon\right)+\alpha\left(\|\varepsilon\|_{H^{1}}^{2}+|b|\right).
    \end{equation*} 
    Combining the above identity with~\eqref{equ:orthe},~\eqref{est:massrough} and the coercivity of the Schr\"odinger operator $\mathcal{L}$ in~\eqref{est:coer}, we complete the proof of~\eqref{est:energyrough}.
        \end{proof}

    In the framework of Lemma~\ref{le:decom}, we introduce the new time variable:
    \begin{equation*}
        s=\int_{0}^{t}\frac{\dd \sigma}{\lambda^{3}(\sigma)}
        \Longleftrightarrow 
        \frac{\dd s}{\dd t}=\frac{1}{\lambda^{3}(t)}.
    \end{equation*}
    In what follows, all functions depending on $(t,x)\in [0,t_{0}]\times \R^{2}$, can be seen as depending on $(s,y)\in [0,s_{0}]\times \R^{2}$ with $s_{0}=s(t_{0})$. In addition, we denote 
    \begin{equation*}
        {\rm{Mod}}=\left(\frac{\lambda_{s}}{\lambda}+b\right)\left(\Lambda Q_{b}+\Lambda \varepsilon\right)
        +\left(\frac{x_{s}}{\lambda}-{\textbf{e}}_{1}\right)\cdot \left(\nabla Q_{b}+\nabla \varepsilon\right)
        -b_{s}\frac{\partial Q_{b}}{\partial b}.
    \end{equation*}
    
    We now recall the equation satisfied by $(\lambda,x,b,\varepsilon)$ under the variables $(s,y)$.
    \begin{lemma}[Standard modulation equations]\label{le:control}
    Under the assumption of Lemma~\ref{le:decom}, and assume moreover
$ \|\nabla \varepsilon(t)\|_{L^{2}}<\kappa<\kappa^{*}$ on $[0,t_{0}]$, for a small enough universal constant $0<\kappa^{*}\ll 1$.  Then the map $(\lambda,x,b):[0,s_{0}]\mapsto (0,\infty)\times\R^{2}\times \R$ is $C^{1}$ and the following hold.
\begin{enumerate}
    \item \emph{(Equation of $\varepsilon$)}. For all $s\in [0,s_{0}]$, we have 
     \begin{equation*}
        \partial_{s}\varepsilon=\partial_{y_{1}}\mathcal{L}\varepsilon
        +\Psi_{b}
        +{\rm{Mod}}
        -b\Lambda \varepsilon
        -\partial_{y_{1}}R_{b}-\partial_{y_{1}}R_{NL},
    \end{equation*}
    where $R_{b}$ and $R_{NL}$ are defined by 
    \begin{equation*}
        R_{b}=3(Q_{b}^{2}-Q^{2})\varepsilon\quad \mbox{and}\quad 
        R_{NL}=3Q_{b}\varepsilon^{2}+\varepsilon^{3}.
    \end{equation*} 
    \item \emph{(Control of parameters)} For all $s\in [0,s_{0}]$, we have 
    \begin{equation*}
    \begin{aligned}
        |b_{s}|&\lesssim b^{2}+\int_{\R^{2}}\varepsilon^{2}e^{-\frac{|y|}{10}}\dd y,\\
        \left|\frac{\lambda_{s}}{\lambda}+b\right|+\left|\frac{x_{s}}{\lambda}-{\textbf{e}}_{1}\right|&\lesssim
        b^{2}+\left(\int_{\R^{2}}\varepsilon^{2}e^{-\frac{|y|}{10}}\dd y\right)^{\frac{1}{2}}.
        \end{aligned}
    \end{equation*}
\end{enumerate}
        \end{lemma}

        \begin{proof}
            The proofs of the equation of $\varepsilon$ and the law of $(\lambda,x,b)$ rely on some elementary computations and the orthogonality condition~\eqref{equ:orthe}. We refer to~\cite[Lemma 3.2 and Lemma 3.6]{CLY} for the details of the proofs.
        \end{proof}

        \begin{remark}
           More precisely, in the previous work~\cite[Lemma 3.6]{CLY}, we assume the smallness of a weighted $H^{1}$ norm for the remainder term instead of the smallness of the $\dot{H}^{1}$ norm. However, such a stronger smallness assumption is only needed in the derivation of the refined control of $(\lambda,x,b)$ (see Lemma~\ref{le:refin} below). This is the main reason why we could relax the assumption in the above statement of Lemma~\ref{le:control}.
        \end{remark}

        To state the refined control of $(\lambda,x,b)$, we consider the weighted $H^{1}$ norm of $\varepsilon$ with specific weight functions. For $i=0,1,2$, we consider the function $\vartheta_{i}\in C^{\infty}(\R)$ with
\begin{equation*}
\begin{aligned}
\vartheta_{i}(y_1)=
\begin{cases}
\frac{1}{2},&\mbox{for}\ y_1\in (-\infty,\frac{1}{2}),\\
y_1^{i+6},&\mbox{for}\ y_1\in(1,+\infty),
\end{cases}
\quad \vartheta_{i}'(y_1)\ge0,\ \ \mbox{for any}\ y_{1}\in\mathbb{R}.
\end{aligned}
\end{equation*}
In addition, we consider the even function $\zeta\in C^{\infty}(\R)$ with $\zeta\in (0,1]$ as follows,
\begin{equation*}
\begin{aligned}
\zeta(y_1)=
\begin{cases}
e^{2y_{1}},&\mbox{for}\ y_1\in(-\infty,-\frac{1}{6}),\\
1,&\mbox{for}\ y_1\in(-\frac{1}{10},\frac{1}{10}),\\
e^{-2y_{1}},&\mbox{for}\ y_1\in(\frac{1}{6},\infty),\\
\end{cases}
\quad
\mbox{and}\ \
\int_{\R}\zeta (y_{1})\dd y_{1}=1.
\end{aligned}
\end{equation*}
Let $B>100$ be a large enough universal constant to be chosen later. For $i=0,1,2$, we consider the following weight function related to $\vartheta_{i}$,
\begin{equation*}
\vartheta_{i,B}(y_1)=\vartheta_{i}\left(\frac{y_1}{B^{10}}\right),\quad \mbox{for any}\ y_{1}\in\R.
\end{equation*} 
We also set a weight function $\psi_B\in C^\infty(\mathbb{R})$ such that 
\begin{equation*}
\lim_{y_1\rightarrow-\infty}\psi_B(y_1)=0 \ \ \mbox{and}\ \ 
\psi_{B}'(y_1)=
\begin{cases}
\frac{1}{B}\zeta\left(\frac{y_{1}}{B}+\frac{1}{3}-\frac{1}{2}B^{-\frac{1}{3}}\right),&\text{ for }y_1<-\frac{1}{3}B,\\
\frac{1}{B}\zeta\left(\frac{y_1}{B^{\frac{2}{3}}}+\frac{1}{3}B^{\frac{1}{3}}\right),&\text{ for }y_1\geq-\frac{1}{3}B.
\end{cases}
\end{equation*}
Next, for $i=0,1,2$, we denote 
\begin{equation*}
\varphi_{i,B}(y_{1})=\sqrt{2\psi_{B}(y_{1})\vartheta^{2}_{i,B}(y_{1})},\quad \mbox{for any}\ y_{1}\in\R.
\end{equation*}

For $i=0,1,2$, we consider the following weighted $H^{1}$ norm of remainder term:
\begin{equation*}
    \mathcal{N}_{i}(s)=\int_{\R^{2}}\left[|\nabla \varepsilon(s,y)|^{2}\psi_{B}(y_{1})+|\varepsilon(s,y)|^{2}\varphi_{i,B}(y_{1})\right]\dd y_{1}\dd {y_{2}}.
\end{equation*}

Recall that, $\phi(t)$ is a solution of~\eqref{equ:CP} which satisfies~\eqref{equ:assumdecom} on $[0,t_{0}]$ and thus on $[0,t_{0}]$ admits a decomposition~\eqref{equ:vare} as in Lemma~\ref{le:decom}. Let $0<\kappa\ll 1$ be a small enough universal constant. 
Denote $s_{0}=s(t_{0})$ and assume the following a priori bounds for any $s\in [0,s_{0}]$:
\begin{enumerate}
    \item [(${\rm{H1}}$)] Smallness. We assume 
    \begin{equation*}
        |b(s)|+\|\varepsilon(s)\|_{L^{2}}+\mathcal{N}_{2}(s)\le \kappa.
    \end{equation*}

    \item [(${\rm{H2}}$)] Bound related to scaling. We assume 
    \begin{equation*}
       \frac{ |b(s)|}{\lambda^{\theta}(s)}
       +\frac{\mathcal{N}_{2}(s)}{\lambda^{\theta}(s)}
       \le \kappa.
    \end{equation*}

    \item [($\rm{H3}$)] $L^{2}$-weighted bound on the right-hand side of $y_{1}$. We assume 
    \begin{equation*}
        \int_{\R}\int_{0}^{\infty}y_{1}^{100}\varepsilon^{2}(s,y)\dd y_{1}\dd y_{2}\le 10\left(1+\frac{1}{\lambda^{100}(s)}\right).
    \end{equation*}
\end{enumerate}

We now recall the refined control of $(\lambda,x,b)$ from~\cite[Lemma 3.7]{CLY}.

        \begin{lemma}[Refined control of parameters]\label{le:refin}
            Under the assumptions of Lemma~\ref{le:control}, and assume moreover {\rm{(H1)--(H3)}}. Then the quantities $\left\{J_{i}\right\}_{i=1}^{3}$ below are well-defined and satisfy the following estimates for any $s\in [0,s_{0}]$.

        \begin{enumerate}
            \item \emph{Refined control of $\lambda$.} Set 
            \begin{equation*}
              J_{1}(s)=\left(\varepsilon(s),\rho_{1}\right)
                \quad \mbox{with}\quad  \rho_{1}(y)=\frac{1}{\|F\|^{2}_{L^{2}}}\int_{-\infty}^{y_{1}}\Lambda Q(\sigma,y_{2})\dd \sigma.
            \end{equation*}
            It holds
            \begin{equation*}
                \left|\frac{\lambda_{s}}{\lambda}+b-2J_{1s}\right|
                \lesssim B^{5}b^{2}+B^{5}\mathcal{N}_{0}.
            \end{equation*}

            \item \emph{Identity of $b$.} We have 
            \begin{equation*}
\begin{aligned}
&b_{s}+\theta b^{2}-\frac{b}{\left(P,Q\right)}\left(\varepsilon,\Lambda Q+6PQ\partial_{y_{1}}Q\right)\\
&-b\left(\frac{\lambda_{s}}{\lambda}+b\right)\frac{\left(\Lambda P, Q\right)}{\left(P,Q\right)}=O\left(|b|^{3}+\int_{\R^{2}}\varepsilon^{2}e^{-\frac{|y|}{10}}\dd y\right).
\end{aligned}
\end{equation*}

            \item \emph{Refined control of $b$.} Set 
            \begin{equation*}
            \begin{aligned}
            \rho_{2}(y)&=\frac{1}{(P,Q)}\left(P(y)+F(y_{2})+h(y_{2})\right)\\
            &+\frac{(\Lambda P,Q)}{(P,Q)(\Lambda Q,Q^{3})}Q^{3}(y)-c_{1}\int_{-\infty}^{y_{1}}\Lambda Q(\sigma,y_{2})\dd \sigma.
                \end{aligned}
            \end{equation*}
            Here, we denote by $h_{2}\in \mathcal{Z}(\R)$ the even solution of the second-order ODE:
            \begin{equation*}
                -h''(y_{2})+h_{2}(y_{2})=F''(y_{2}),\quad \mbox{for any}\ y_{2}\in \R.
            \end{equation*}
            In addition, the constant $c_{1}\in \R$ is chosen to ensure that 
            \begin{equation*}
               c_{1}\left(\int_{-\infty}^{y_{1}}\Lambda Q(\sigma,y_{2})\dd \sigma,\Lambda Q\right)=\frac{1}{(P,Q)}\left(F+h_{2},\Lambda Q\right).
            \end{equation*}
            Denote 
            \begin{equation*}
                J_{2}(s)=\left(\varepsilon(s),\rho_{2}\right)
                \quad \mbox{and}\quad 
                 g_{2}(y)=\partial_{y_{1}}\rho_{2}(y).
            \end{equation*}
            It holds
            \begin{equation*}
                \left|b_{s}+\theta b^{2}+bJ_{2s}\right|\lesssim
                B^{5}|b|^{3}+\left(B^{5}|b|+1\right)\mathcal{N}_{0}.
            \end{equation*}

            \item \emph{Refined control of $\frac{b}{\lambda^{\theta}}$.} Set 
            \begin{equation*}
             J(s)=\left(\varepsilon(s),\rho\right)\quad 
             \mbox{and}\quad 
                \rho=2\theta\rho_{1}+\rho_{2}.
            \end{equation*}
            It holds
            \begin{equation*}
                \left|\frac{\dd }{\dd s}\left(\frac{b}{\lambda^{\theta}}\right)+\frac{b}{\lambda^{\theta}}J_{s}\right|
                \lesssim \frac{1}{\lambda^{\theta}}
                \left( B^{5}|b|^{3}+\left(B^{5}|b|+1\right)\mathcal{N}_{0}\right).
            \end{equation*}

            \item \emph{Refined control of $x_{2}$}. Set 
            \begin{equation*}
                \rho_{3}(y)=\frac{1}{c_{2}}\int_{-\infty}^{y_{1}}\partial_{y_{2}}Q(\sigma,y_{2})\dd \sigma\quad 
                \mbox{and}\quad J_{3}(s)=\left(\varepsilon(s),\rho_{3}\right).
            \end{equation*}
            Here, the constant $c_{2}\in \R$ is chosen to ensure that 
            \begin{equation*}
                c_{2}=\frac{1}{2}\int_{\R}
                \left(\int_{\R}\partial_{y_{2}}Q(y_{1},y_{2})\dd y_{1}\right)^{2}\dd y_{2}.
            \end{equation*}
            It holds
            \begin{equation*}
                \left|\frac{x_{2s}}{\lambda}-J_{3s}\right|\lesssim  B^{5}b^{2}+B^{5}\mathcal{N}_{0}.
            \end{equation*}
        \end{enumerate}
        \end{lemma}

        \begin{proof}
            The proof of the refined control for $(\lambda,x,b)$ relies on the standard modulation equation in Lemma~\ref{le:control} and the more refined derivation. We refer to~\cite[Lemma 3.7]{CLY} for the details of the proof.
        \end{proof}

      \subsection{Energy-Virial Lyapunov functional}\label{SS:Energy}
      In this subsection, we recall the energy-virial Lyapunov functional $\mathcal{M}_{ij}$ which will play an important role in closing the energy estimate under the bootstrap assumptions (H1)--(H3).

      \smallskip
      Recall that, for $(i,j)\in \left\{1,2\right\}$, we set 
      \begin{equation*}
          \mathcal{J}_{ij}=(1-J_{1})^{-2\theta(j-1)-2i-12}-1.
      \end{equation*}
      Here, the function $J_{1}$ is defined by 
      \begin{equation*}
      J_{1}(s)=\left(\varepsilon(s),\rho_{1}\right)\quad \mbox{with}\quad 
          \rho_{1}(y)=\frac{1}{\|F\|_{L^{2}}^{2}}\int_{-\infty}^{y_{1}}\Lambda Q(\sigma,y_{2})\dd \sigma.
      \end{equation*}

      Recall also that, for $(i,j)\in \left\{1,2\right\}$, we consider the following energy functional related to the remainder term $\varepsilon$,
      \begin{equation*}
          \mathcal{F}_{ij}=\int_{\R^{2}}\left\{|\nabla \varepsilon|^{2}\psi_{B}+(1+\mathcal{J}_{ij})\varepsilon^{2}\varphi_{i,B}
          -\frac{1}{2}\psi_{B}\left[(Q_{b}+\varepsilon)^{4}-Q_{b}^{4}-4Q^{3}_{b}\varepsilon\right]\right\}\dd y.
      \end{equation*}

      To state the virial estimate for the solution of~\eqref{equ:CP}, we should first consider the following transformed problem
      \begin{equation*}
          \eta=(1-\gamma\Delta)^{-1}\mathcal{L}\varepsilon,\quad \mbox{for any}\ s\in [0,s_{0}].
      \end{equation*}
      Here, $0<\gamma\ll 1$ is a small enough constant (depending on $B$) to be chosen later.
      In addition, we denote 
      \begin{equation*}
          \begin{aligned}
              {\rm{Mod}}_{\eta}
              &=\left(\frac{\lambda_{s}}{\lambda}+b\right)\left(\Lambda Q_{b}-\Lambda Q\right)+\frac{\lambda_{s}}{\lambda}\Lambda \varepsilon\\
              &+\left(\frac{x_{s}}{\lambda}-{\textbf{e}}_{1}\right)\cdot
              \left(\nabla Q_{b}-\nabla Q+\nabla \varepsilon\right)-b_{s}\frac{\partial Q_{b}}{\partial s}.
          \end{aligned}
      \end{equation*}

      By an elementary computation and Lemma~\ref{le:control}, we directly have the equation and orthogonality conditions for $\eta$.
      \begin{lemma}
          For all $s\in [0,s_{0}]$, we have 
          \begin{equation*}
\begin{aligned}
\partial_{s}\eta
&=\mathcal{L}\partial_{y_{1}}\eta-3\gamma(1-\gamma\Delta)^{-1}\left(\Delta(Q^{2})\partial_{y_{1}}\eta+2Q\nabla Q\cdot \nabla \partial_{y_{1}}\eta\right)\\
&-2\left(\frac{\lambda_{s}}{\lambda}+b\right)(1-\gamma\Delta)^{-1}Q
+(1-\gamma\Delta)^{-1}{\mathcal{L}}{\rm{Mod}}_{\eta}
\\
&+(1-\gamma\Delta)^{-1}\left({\mathcal{L}}\Psi_{b}-{\mathcal{L}}\partial_{y_{1}}R_{b}-{\mathcal{L}}\partial_{y_{1}}R_{NL}\right).
\end{aligned}
\end{equation*}
In addition, the function $\eta$ satisfies the following orthogonality conditions
\begin{equation*}
\left(\eta,(1-\gamma\Delta)Q\right)=
\left|\left(\eta,(1-\gamma\Delta)\nabla Q\right)\right|=0.
\end{equation*}
      \end{lemma}
     
      Recall that, we consider the function $\widetilde{\chi}\in C^{\infty}$ with $\widetilde{\chi}\in [0,1]$ such that
\begin{equation*}
\widetilde{\chi}(y_1)=
\begin{cases}
0,&\text{ for } |y_{1}|>2,\\
1,&\text{ for } |y_{1}|\le 1.
\end{cases}
\end{equation*}
In addition, we consider the function $\psi_{0}\in C^{\infty}$ with $\psi_{0}\in (0,1]$ such that
\begin{equation*}
\psi_{0}(y_1)=
\begin{cases}
e^{6y_1},&\text{ for }y_1<-1,\\
\frac{1}{2},&\text{ for }y_1>-\frac{1}{2},
\end{cases}
\quad \mbox{with}\ \  \psi_0'(y_1)\ge 0,\ \ \mbox{for any}\ y_{1}\in\mathbb{R}.
\end{equation*}
Let $B>100$ be a large enough universal constant to be chosen later. We denote
\begin{equation*}
\psi_{0,B}(y_1)=\psi_0\left(\frac{y_1}{B}\right),\quad \mbox{for any}\ y_{1}\in \R.
\end{equation*}
Last, we denote 
\begin{equation*}
\widetilde{\chi}_{B}(y_1)=
\begin{cases}
\widetilde{\chi}\left(\frac{y_1}{2B}\right)\int_0^{y_1}\frac{2}{B}\psi_{0,B}(\sigma)\dd \sigma,\quad\quad  \mbox{for}\ y_{1}\le 0,\\
\widetilde{\chi}\left(\frac{y_1}{10B^{10}}\right)\int_0^{y_1}\frac{2}{B}\psi_{0,B}(\sigma)\dd \sigma,\quad \mbox{for}\ y_{1}> 0.
\end{cases}
\end{equation*}

We now define the energy-virial Lyapunov functional $\mathcal{M}_{ij}$ related to the remainder terms $(\varepsilon,\eta)$. More precisely, for $(i,j)\in\{1,2\}$, we set 
\begin{equation*}
    \mathcal{M}_{ij}=\mathcal{F}_{ij}+\frac{1}{B^{20}}\mathcal{P}\quad \mbox{with}\ \ 
    \mathcal{P}=\int_{\R^{2}}\eta^{2}\widetilde{\chi}_{B}\dd y.
\end{equation*}
Let $\gamma=B^{-3}>0$. Then the following estimates hold.
\begin{proposition}\label{prop:energymoni}
    Let $B>100$ be a large enough constant and $0<\kappa_{1}<B^{-100}$ be a small enough constant. Under the assumptions of Lemma~\ref{le:control}, and assume moreover {\rm{(H1)--(H3)}}. The quantities $\{\mathcal{M}_{ij}\}_{i,j=1}^{2}$ above satisy the following estimates.

    \begin{enumerate}
        \item \emph{Bound and Coercivity}. For any $s\in [0,s_{0}]$, we have 
        \begin{equation*}
          \mathcal{N}_{i}\lesssim\mathcal{M}_{ij}\lesssim \mathcal{N}_{i}.
        \end{equation*}
        \item \emph{Monotonicity property.} For any $s\in [0,s_{0}]$, we have
        \begin{equation*}
            \lambda^{\theta(j-1)}\frac{\dd}{\dd s}\left(\frac{\mathcal{M}_{ij}}{\lambda^{\theta(j-1)}}\right)
            +\frac{\nu}{B^{27}}\mathcal{N}_{i-1}\le Cb^{4}.
        \end{equation*}
        Here, $0<\nu\ll 1$ is a universal constant independent with $B$ and $C=C(B)>1$ is a constant dependent only on $B$.
    \end{enumerate}
\end{proposition}

\begin{proof}
    The proof of the bound and coercivity relies on the definition of $\mathcal{M}_{ij}$ and the coercivity of the operator $\mathcal{L}$ in~\eqref{est:coer}. The proof of the monotonicity property can be split into two steps. The first one is the standard energy estimate for $\mathcal{F}_{ij}$ which is only related to $\varepsilon$, and the second one is the virial estimate for $\mathcal{P}$ which is only related to $\eta$. It is worth mentioning here that the virial estimate is essentially based on the coercivity of the operator $\mathcal{A}$ in~\eqref{est:coer2} which has been numerically verified. We refer to~\cite[Proposition 4.12]{CLY} for the details of the proof.
\end{proof}

      \subsection{Almost monotonicity of the mass}
In this subsection, we deduce the almost monotonicity of the mass on the right-hand side of $y_{1}$. 

\smallskip
First, we recall the following 2D weighted Sobolev estimate for future reference.

\begin{lemma}\label{le:weightedSobolev}
    Let $\omega:\R^{2}\to(0,\infty)$ be a $C^{1}$ function such that $\|\nabla \omega/\omega\|_{L^{\infty}}\lesssim 1$. In addition, for any $R>0$, we denote 
    \begin{equation*}
        B_{R}=\left\{(x_{1},x_{2})\in \R^{2}:|x_{1}|>R\ \emph{or}\ |x_{2}|>R\right\}.
    \end{equation*}
    Then for any $f\in H^{1}(\R^{2})$, we have 
    \begin{equation*}
        \|f^{2}\sqrt{\omega}\|^{2}_{L^{2}(B_{R})}
        \lesssim \|f\|_{L^{2}(B_{R})}^{2}\left(\int_{B_{R}}(|\nabla f|^{2}+f^{2})\omega \dd x\right).
    \end{equation*}
\end{lemma}
\begin{proof}
    The proof relies on a standard argument based on the Fundamental Theorem and the 1D weighted Sobolev estimate. We refer to~\cite[Lemma 4.3]{CLY} and~\cite[Lemma 6.1]{FHRY} for the details of the proof.
\end{proof}

For $A\gg 1$ large enough to be chosen later, we set
\begin{equation*}
    \psi(x_{1})=\frac{2}{\pi}\arctan \left(e^{-x_{1}}\right)\quad\mbox{and}\quad 
    \psi_{A}(x_{1})=\psi\left(\frac{x_{1}}{A}\right)\quad 
    \mbox{on}\ \R.
\end{equation*}
Note that, from the definition of $\psi_{A}$, we see that 
\begin{equation}\label{est:pointpsi}
    \psi'_{A}<0,\quad  \left|\psi''_{A}\right|\lesssim \frac{1}{A}\left|\psi'_{A}\right|\quad \mbox{and}\quad 
    \left|\psi'''_{A}\right|\lesssim \frac{1}{A^{2}}\left|\psi'_{A}\right|,\quad \mbox{on}\ \R.
\end{equation}
In addition, for any $(t_{1},x_{0})\in [0,t_{0}]\times (0,\infty)$, we define
\begin{equation*}
    \mathcal{I}_{t_{1},x_{0}}(t)=\int_{\R^{2}}\phi^{2}(t,x)\psi_{A}\left(x_{1}-x_{0}-x_{1}(t_{1})-\frac{1}{4}\left(x_{1}(t)-x_{1}(t_{1})\right)\right)\dd x_{1}\dd x_{2}.
\end{equation*}

Let us now recall the almost monotonicity of the mass on the right-hand side of $y_{1}$ from~\cite{FHRY}. The proof is similar to~\cite[Lemma 6.2]{FHRY}
and~\cite[Lemma 4 and Lemma 6]{MMDUKE}, but it is given for the sake of completeness and the reader's convenience.
\begin{lemma}[\cite{FHRY}]\label{le:monomass}
Under the assumptions of Lemma~\ref{le:control}, and assume moreover $0<\lambda(t)<\frac{3}{2}$ for any $t\in [0,t_{1}]$. Then there exists $A\gg 1$ such that for any $x_{0}>0$, 
\begin{equation*}
    \mathcal{I}_{t_{1},x_{0}}(t)\le \mathcal{I}_{t_{1},x_{0}}(t_{1})+Ae^{-\frac{x_{0}}{A}},\quad \mbox{for any}\ \  t\in [0,t_{1}].
\end{equation*}
    Here, $A$ is a universal large constant independent with $\phi(t)$.
\end{lemma}
\begin{proof}
To simplify notation, we denote 
\begin{equation*}
    \widetilde{x}_{1}=x_{1}-x_{0}-x_{1}(t_{1})-\frac{1}{4}\left(x_{1}(t)-x_{1}(t_{1})\right).
\end{equation*}
In addition, for $R\gg 1$ large enough to be chosen later, we denote
    \begin{equation*}
    \begin{aligned}
        \mathcal{B}_{1R}&=\left\{(x_{1},x_{2})\in \mathbb{R}^{2}:|x_{1}-x_{1}(t)|\ge R\ \mbox{or}\ |x_{2}-x_{2}(t)|\ge R\right\},\\
         \mathcal{B}_{2R}&=\left\{(x_{1},x_{2})\in \mathbb{R}^{2}:|x_{1}-x_{1}(t)|\le R\ \mbox{and}\ |x_{2}-x_{2}(t)|\le R\right\}.
        \end{aligned}
    \end{equation*}

{\it Step 1. First estimate on $\mathcal{I}_{t_{1},x_{0}}$.} We claim that 
\begin{equation}\label{est:I1}
     \mathcal{I}_{t_{1},x_{0}}(t)-\mathcal{I}_{t_{1},x_{0}}(t_{1})
     \lesssim\int_{t}^{t_{1}}\int_{\mathcal{B}_{2R}}\phi^{4}
     \left|\psi'_{A}(\widetilde{x}_{1})\right|\dd x_{1}\dd x_{2}.
\end{equation}
   Indeed, using~\eqref{equ:CP} and integration by parts, we compute 
    \begin{equation*}
    \begin{aligned}
        \frac{\dd }{\dd t}\mathcal{I}_{t_{1},x_{0}}(t)
        &=-\int_{\R^{2}}\left(3(\partial_{x_{1}}\phi)^{2}+(\partial_{x_{2}}\phi)^{2}\right)\psi'_{A}(\widetilde{x}_{1})\dd x_{1}\dd x_{2}\\
        &+\int_{\R^{2}}\phi^{2}\left(\psi'''_{A}(\widetilde{x}_{1})-\frac{x_{1t}}{4}\psi'_{A}(\widetilde{x}_{1})\right)\dd x_{1}\dd x_{2}
        +\frac{3}{2}\int_{\R^{2}}\phi^{4}\psi'_{A}(\widetilde{x}_{1})\dd x_{1}\dd x_{2}.
        \end{aligned}
    \end{equation*}
    First, from Lemma~\ref{le:decom}, Lemma~\ref{le:control} and $0<\lambda(t)<\frac{3}{2}$, we have 
    \begin{equation*}
       \left| \lambda^{2}x_{1t}-1\right|\lesssim \alpha(\kappa)\Longrightarrow
       x_{1t}\ge \frac{1-\alpha(\kappa)}{\lambda^{2}}\ge \frac{1}{3}.
    \end{equation*}
    Therefore, for $A\gg 1$ large enough, we deduce that 
    \begin{equation*}
        \int_{\R^{2}}\phi^{2}\left(\psi'''_{A}(\widetilde{x}_{1})-\frac{x_{1t}}{4}\psi'_{A}(\widetilde{x}_{1})\right)\dd x_{1}\dd x_{2}\ge -\frac{1}{20}\int_{\R^{2}}\phi^{2}\psi'_{A}(\widetilde{x}_{1})\dd x_{1}\dd x_{2}.
    \end{equation*}
    Recall that, we have the following decomposition,
    \begin{equation*}
        \lambda(t)\phi(t,\lambda(t)y+x(t))=Q_{b(t)}(y)+\varepsilon(t,y).
    \end{equation*}
   Therefore, from Lemma~\ref{le:decom}, the exponential decay of $Q$ and $0<\lambda(t)<\frac{3}{2}$,
    \begin{equation*}
        \|\phi\|^{2}_{L^{2}(\mathcal{B}_{1R})}
        \lesssim  \alpha(\kappa)+e^{-\frac{R}{\lambda(t)}}
        \lesssim \alpha(\kappa)+e^{-\frac{R}{4}}.
    \end{equation*}
    Based on the above estimate and Lemma~\ref{le:weightedSobolev}, we obtain
    \begin{equation*}
    \left(\alpha(\kappa)+e^{-\frac{R}{4}}\right)
        \int_{\R^{2}}\left(|\nabla \phi|^{2}+\phi^{2}\right)\psi'_{A}(\widetilde{x}_{1})\dd x_{1}\dd x_{2}\lesssim
         \int_{\mathcal{B}_{1R}}\phi^{4}\psi'_{A}(\widetilde{x}_{1})\dd x_{1}\dd x_{2}.
    \end{equation*}
    Combining the above estimates, we obtain 
    \begin{equation*}
        \frac{\dd }{\dd t}\mathcal{I}_{t_{1},x_{0}}(t)\gtrsim 
         \int_{\mathcal{B}_{2R}}\phi^{4}\psi'_{A}(\widetilde{x}_{1})\dd x_{1}\dd x_{2},\quad \mbox{on}\ [0,t_{1}].
    \end{equation*}
    Integrating the above estimate over $[t,t_{1}]$, we complete the proof of~\eqref{est:I1}.

    \smallskip
{\it    Step 2. Conclusion.} Note that, for any $x_1\in \mathcal{B}_{2R}$, we have 
    \begin{equation*}
        \left|\widetilde{x}_{1}\right|=
        \left|x_{1}-x_{1}(t)-x_{0}+\frac{1}{4}\left(x_{1}(t)-x_{1}(t_{1})\right)\right|
        \ge x_{0}+\frac{1}{4}\left(x_{1}(t_{1})-x_{1}(t)\right)-R.
    \end{equation*}
    It follows from the definition of $\psi_{A}$ that 
    \begin{equation*}
        \left\|\psi'_{A}(\widetilde{x}_{1})\right\|_{L^{\infty}(\mathcal{B}_{2R})}
        \lesssim 
        \frac{1}{A}\exp \left(-\frac{x_{0}}{A}+\frac{R}{A}
        -\frac{1}{4A}\left(x_{1}(t_{1})-x_{1}(t)\right)
        \right).
    \end{equation*}
    Based on the above estimate and the fact that $\|\nabla \phi(t)\|^{2}_{L^{2}}\sim \lambda^{-2}\lesssim x_{1t}$, we have
    \begin{equation*}
    \begin{aligned}
       \left| \int_{\mathcal{B}_{2R}}\phi^{4}\psi'_{A}(\widetilde{x}_{1})\dd x_{1}\dd x_{2}\right|
       &\lesssim \left\|\psi'_{A}(\widetilde{x}_{1})\right\|_{L^{\infty}
       (\mathcal{B}_{2R})}x_{1t}\\
       &\lesssim
        \frac{1}{A}\exp \left(-\frac{x_{0}}{A}+\frac{R}{A}
        -\frac{1}{4A}\left(x_{1}(t_{1})-x_{1}(t)\right)
        \right)x_{1t}.
       \end{aligned}
    \end{equation*}
    Integrating the above estimate over $[t,t_{1}]$, we complete the proof of Lemma~\ref{le:monomass}.
\end{proof}

\section{End of the proof for Theorem~\ref{thm:main}}\label{S:End}
\subsection{Main properties of minimal mass solutions}\label{SS:Mini}
In this subsection, we introduce the main properties related to the minimal mass blow-up solution. In what follows, we will consider the blow-up solution with a negative blow-up time. Due to the invariance property under time reversal for equation~\eqref{equ:CP}, the statement of Theorem~\ref{thm:main} for the cases of forward time and backward time are equivalent. 

\begin{remark}
Inspired by the previous work~\cite[Lemma 2.11]{MMR1}, we will show the exponential decay on the right side of $y_{1}$ for the minimal mass blow-up solution with negative blow-up time (see more details in Proposition~\ref{prop:decay}). Based on this, the minimal mass blow-up solution will enter the monotonicity regime and thus we could use the law of $(\lambda,b)$ and the monotonicity property of energy-virial Lyapunov functional that established in Section~\ref{SS:Geome} and Section~\ref{SS:Energy} to control the minimal mass blow-up solution. This is the main reason why we consider the backward-in-time evolution of~\eqref{equ:CP} at the end of the proof for Theorem~\ref{thm:main}.
\end{remark}

We now introduce the following standard geometric decomposition property related to the minimal mass blow-up solution which blows up in backward time.

\begin{proposition}[Decomposition property]\label{pro:decommini}
    Let $\phi(t)$ be a minimal mass blow-up solution which blows up backward in time $-\infty\le T<0$. Then there exists $t_{1}\in (T,0)$ close to $T$, such that for all $t\in (T,t_{1}]$, $\phi(t)$ or $-\phi(t)$ admits a decomposition $\left(\lambda(t),b(t),x(t),\varepsilon(t)\right)$ as in Lemma~\ref{le:decom} with
    \begin{equation*}
    \lambda(t)\to 0 \quad \mbox{as}\ t\downarrow T\quad \mbox{and}\quad 
    \lambda(t)\le \lambda(t_{1})\ \mbox{on}\ (T,t_{1}].
    \end{equation*}
    Moreover, for any $t\in (T,t_{1}]$, we have
    \begin{equation*}
    b(t)<0,\quad E(\phi(t))=E_{0}>0\quad \mbox{and}\quad     |b(t)|+\|\varepsilon(t)\|_{H^{1}}^{2}\lesssim \lambda^{2}(t)E_{0}.
    \end{equation*}
\end{proposition}

\begin{proof}
First, from the definition of the minimal mass blow-up solution and Lemma~\ref{le:varia}, we see that either $\phi(t)$ or $-\phi(t)$ satisfies~\eqref{equ:assumdecom} for $t$ close to $T$, with in addition 
\begin{equation*}
    \|\overline{\varepsilon}(t)\|_{H^{1}}\to 0,\quad \mbox{as}\ t\downarrow T.
\end{equation*}
Therefore, from Lemma~\ref{le:decom}, there exists $t_{1}$ close to $T$ such that the function $\phi(t)$ (or $-\phi(t)$) admits a decomposition on $(T,t_{1}]$:
\begin{equation*}
    \phi(t,x)=\frac{1}{\lambda(t)}\left(Q_{b(t)}+\varepsilon\right)\left(t,\frac{x-x(t)}{\lambda}\right).
\end{equation*}
Using again the definition of minimal mass blow-up solution and the smallness estimates in Lemma~\ref{le:decom}, we deduce that 
\begin{equation*}
    \|\nabla \phi(t)\|_{L^{2}}\lesssim \frac{1}{\lambda(t)}\Longrightarrow
    \lambda(t)\to 0,\quad \mbox{as}\ t\downarrow T.
\end{equation*}
Therefore, for $t_{1}$ close to $T$ enough, we directly have  $\lambda(t)\le \lambda(t_{1})$ on $(T,t_{1}]$.

Next, from (ii) of Lemma~\ref{le:blambda} and the fact that $\phi(t)$ is not a soliton, we deduce that 
\begin{equation*}
    E(\phi(t))=E_{0}>0,\quad \mbox{for any}\ t\in (T,t_{1}].
\end{equation*}
Hence, using (i) and (ii) of Lemma~\ref{le:blambda}, we complete the proof of Proposition~\ref{pro:decommini}.
\end{proof}

In the setting of Proposition~\ref{pro:decommini}, we now introduce the following rescaled time:
\begin{equation}\label{equ:defins}
    s(t)=-\int_{t}^{t_{1}}\frac{\dd \sigma}{\lambda^{3}(\sigma)},\quad \mbox{for any}\ t\in (T,t_{1}].
\end{equation}
Note that, even though we consider the backward-in-time evolution of~\eqref{equ:CP}, from the definition of the new time variable $s$ in~\eqref{equ:defins}, the estimates which was established in Section~\ref{SS:Geome} and Section~\ref{SS:Energy} still hold.

\begin{remark}
Using a scaling argument and the resolution of the Cauchy problem, if $\phi(t)$ is a solution blowing up at some finite negative time $T$, then we have 
\begin{equation}\label{est:lowerbound}
    \|\nabla \phi(t)\|_{L^{2}}\gtrsim \frac{1}{(t-T)^{\frac{1}{3}}},\quad \mbox{for any}\ t\in (T,0].
\end{equation}
Indeed, for any fixed $t\in (T,0]$ close to $T$, we consider 
\begin{equation*}
    v(\tau,x)=\frac{1}{\|\nabla \phi(t)\|_{L^{2}}}\phi\left(t+\frac{\tau}{\|\nabla \phi(t)\|^{3}_{L^{2}}},\frac{x}{\|\nabla \phi(t)\|_{L^{2}}   }\right).
\end{equation*}
By an elementary computation, we deduce that $v(\tau,x)$ is a solution of~\eqref{equ:CP} under the variable $(\tau,x)$ and satisfies $\|v(0)\|_{H^{1}}\lesssim 1$. By the local well-posedness for~\eqref{equ:CP}, there exists $\tau_{0}<0$ independent with $t$ such that $v(\tau)$ is well-defined on $[\tau_{0},0]$ and thus we have $T<t+\|\nabla \phi(t)\|_{L^{2}}^{-3}\tau_{0}$ which is the desired result. We mention here that, from the defintion of $s$ in~\eqref{equ:defins} and the lower bound of blow-up rate in~\eqref{est:lowerbound}, we obtain $s(T)=-\infty$ for the both cases of $T$ finite and infinite. 
\end{remark}

Last, for the minimal mass blow-up solution, we establish the exponential decay for the one-dimensional integration over the right-hand side of $y_{1}$.
\begin{proposition}[Decay property]\label{prop:decay}
In the context of Proposition~\ref{pro:decommini}, we have 
\begin{equation*}
   \sup_{T<t\le t_{1}}\int_{\R}\varepsilon^{2}(t,y_{1},y_{2})\dd y_{2}\le A_{1}e^{-\frac{y_{1}}{A_{1}}},\quad \mbox{for any}\ y_{1}>0.
\end{equation*}
Here, $A_{1}>1$ is a universal constant independent with $\phi(t)$.
\end{proposition}

\begin{proof} Following the same procedure as finding the $t_1$ in Proposition~\ref{pro:decommini}, we can construct a decreasing sequence of time $\left\{t_{n}\right\}_{n=1}^{\infty}$ such that
\begin{equation*}
    t_{n}\downarrow T\ \ \mbox{as}\ n\to\infty\quad \mbox{and}\quad 
    \lambda(t)\le \lambda(t_{n})\ \ \mbox{on}\ (T,t_{n}].
\end{equation*}
{\it Step 1. Exponential decay over the sequence of time.} We claim that, 
\begin{equation}\label{est:expn}
    \sup_{n\in \mathbb{N}}\int_{\R^{2}}\varepsilon^{2}\left(t_{n},y_{1},y_{2}\right)\dd y_{2}\le A_{2}e^{-\frac{y_{1}}{A_{2}}},\quad\mbox{for any}\ y_{1}>0.
\end{equation}
Here, $A_{2}>1$ is a universal constant independent with $\phi(t)$.

\smallskip
Indeed, we first denote the following sequence of the time variable $s$:
\begin{equation*}
    s_{n}=s(t_{n})=-\int_{t_{n}}^{t_{1}}\frac{\dd \sigma}{\lambda^{3}(\sigma)},\quad \mbox{for any}\ n\in \mathbb{N}^{*}.
\end{equation*}
We consider the following renormalization of $\phi(t)$:
\begin{equation*}
    \phi_{n}(t',x')=\lambda(s_{n})\phi(\lambda^{3}(s_{n})t',\lambda(s_{n})x'),\ \ \mbox{for any}\ (t',x')\in \left(\frac{T}{\lambda^{3}(s_{n})},\frac{t_{1}}{\lambda^{3}(s_{n})}\right)\times \R^{2}.
\end{equation*}
We denote by $(\lambda_{n},b_{n},x_{n},\varepsilon_{n})$ the geometric parameters and the remainder term associated with the decomposition of $\phi_{n}$ via  Lemma~\ref{le:decom}.
Then, from the definition of decomposition and the construction of the decreasing sequence of time $\left\{t_{n}\right\}_{n=1}^{\infty}$, we have for $n\in\mathbb{N}^*$
\begin{equation*}
    \lambda_{n}(s)=\frac{\lambda(s)}{\lambda(s_{n})},\quad  
    x_{n}(s)=\frac{x(s)}{\lambda(s_{n})}\quad \mbox{and}\quad 
    0<\lambda_{n}(s)\le 1,\quad \mbox{for any}\ s\in (-\infty,s_{n}].
\end{equation*}

Fixing a  $y_{0}>0$ and applying Lemma~\ref{le:monomass} to $\phi_{n}(s)$ over $[s,s_{n}]$ with $x_{0}=y_{0}$, we directly have
\begin{equation*}
\begin{aligned}
    &\int_{\R^{2}}\phi_{n}^{2}(s,x'_{1},x'_{2})\psi_{A}\left(x'_{1}-y_{0}-x_{n1}(s_{n})-\frac{1}{4}\left(x_{n1}(s)-x_{n1}(s_{n})\right)\right)\dd x'_{1}\dd x'_{2}\\
    &\le \int_{\R^{2}}\phi_{n}^{2}(s_{n},x'_{1},x'_{2})\psi_{A}\left(x'_{1}-y_{0}-x_{n1}(s_{n})\right)\dd x'_{1}\dd x'_{2}+Ae^{-\frac{y_{0}}{A}}.
    \end{aligned}
\end{equation*}
On the one hand, we denote $v(s,y)=Q_{b(s)}(y)+\varepsilon(s,y)=\lambda(s)\phi(s,\lambda(s)y+x(s))$. Then, for any $n\in \mathbb{N}^{*}$, we have $v(s_{n},y)=\phi_{n}(s_{n},y+x_{n}(s_{n}))$ and 
\begin{equation*}
\begin{aligned}
   &\int_{\R^{2}}v^{2}(s_{n},y)\psi_{A}\left(y_{1}-y_{0}\right)\dd y_{1}\dd y_{2}\\
   &=\int_{\R^{2}}\phi_{n}^{2}(s_{n},x'_{1},x'_{2})\psi_{A}\left(x'_{1}-y_{0}-x_{n1}(s_{n})\right)\dd x'_{1}\dd x'_{2}.
   \end{aligned}
\end{equation*}
Then, from $(x_{n1})_s>0$ and the function $\psi_{A}$ is decreasing, we find
\begin{equation*}
\begin{aligned}
&\int_{\R^{2}}\phi_{n}^{2}(s,x'_{1},x'_{2})\psi_{A}\left(x'_{1}-y_{0}-x_{n1}(s)\right)\dd x'_{1}\dd x'_{2}\\
&\le \int_{\R^{2}}\phi_{n}^{2}(s,x'_{1},x'_{2})\psi_{A}\left(x'_{1}-y_{0}-x_{n1}(s_{n})-\frac{1}{4}\left(x_{n1}(s)-x_{n1}(s_{n})\right)\right)\dd x'_{1}\dd x'_{2}.
    \end{aligned}
\end{equation*}
Combining the above identities and estimates, we obtain 
\begin{equation}\label{est:exp1}
    \begin{aligned}
        &\int_{\R^{2}}\phi_{n}^{2}(s,x'_{1},x'_{2})\psi_{A}\left(x'_{1}-y_{0}-x_{n1}(s)\right)\dd x'_{1}\dd x'_{2}\\
        &\le \int_{\R^{2}}v^{2}(s_{n},y)\psi_{A}\left(y_{1}-y_{0}\right)\dd y_{1}\dd y_{2}+Ae^{-\frac{x_{0}}{A}}.
    \end{aligned}
\end{equation}
On the other hand, from the definition of $\phi_{n}$ and $\lambda(t)\to 0$ as $t\downarrow T$, we deduce that 
\begin{equation*}
    \phi^{2}_{n}(s,x'+x_{n}(s))\rightharpoonup \left(\int_{\R^{2}}Q^{2}\dd y\right)\delta_{x=0},\quad \mbox{as}\ s\downarrow -\infty.
\end{equation*}
Therefore, as $s\to -\infty$ in~\eqref{est:exp1}, we obtain 
\begin{equation*}
   \left( \int_{\R^{2}}Q^{2}\dd y\right)\psi_{A}(-y_{0})
   \le\int_{\R^{2}}v^{2}(s_{n},y)\psi_{A}\left(y_{1}-y_{0}\right)\dd y_{1}\dd y_{2}+Ae^{-\frac{x_{0}}{A}}.
\end{equation*}
Based on the above estimate and the definition of $\psi_{A}$, 
\begin{equation*}
\begin{aligned}
\int_{\R}\int_{x_{0}}^{\infty}v^{2}(s_{n},y)\dd y_{1}\dd y_{2}
&\le 2 \int_{\R^{2}}v^{2}(s_{n},y)\left(1-\psi_{A}\left(y_{1}-y_{0}\right)\right)\dd y_{1}\dd y_{2}\\
&\le 2\left(\int_{\R^{2}}Q^{2}\dd y\right)\left(1-\psi_{A}(-y_{0})\right)+2Ae^{-\frac{y_{0}}{A}}\lesssim e^{-\frac{y_{0}}{A}}.
\end{aligned}
\end{equation*}
From the exponential decay of $Q_{b}$ on the right side of $y_{1}$, we directly have 
\begin{equation}\label{est:exp2}
\begin{aligned}
    \int_{\R}\int_{y_{0}}^{\infty}\varepsilon^{2}(s_{n},y)\dd y_{1}\dd y_{2}
    &\le 2 \int_{\R}\int_{y_{0}}^{\infty}v^{2}(s_{n},y)\dd y_{1}\dd y_{2}\\
    &+2 \int_{\R}\int_{y_{0}}^{\infty}Q^{2}_{b(s_{n})}(y)\dd y_{1}\dd y_{2}\le A_{3}e^{-\frac{y_{0}}{A_{3}}}.
    \end{aligned}
\end{equation}
Here, $A_{3}$ is a universal constant independent of $\phi(t)$ and $y_{0}$.

\smallskip
Last, from the Fundamental Theorem and the Cauchy-Schwarz inequality,  
\begin{equation*}
    \varepsilon^{2}(s_{n},y_{1},y_{2})\lesssim \left(\int_{y_{1}}^{\infty}\varepsilon^{2}(s_{n},\sigma,y_{2})\dd \sigma\right)^{\frac{1}{2}}\left(\int_{y_{1}}^{\infty}\left|\nabla \varepsilon\right|^{2}(s_{n},\sigma,y_{2})\dd \sigma\right)^{\frac{1}{2}}.
\end{equation*}
Integrating the above estimate under the variable $y_{2}$ and then using~\eqref{est:exp2},  the Cauchy-Schwarz inequality results in
\begin{equation*}
   \int_{\R} \varepsilon^{2}(s_{n},y_{1},y_{2})\dd y_{2}
   \lesssim \left( \int_{\R}\int_{y_{0}}^{\infty}\varepsilon^{2}(s_{n},y)\dd y_{1}\dd y_{2}\right)^{\frac{1}{2}}\|\nabla \varepsilon(s_{n})\|_{L^{2}}
   \lesssim e^{-\frac{y_{1}}{2A_{3}}},
\end{equation*}
which completes the proof of~\eqref{est:expn} for any $n\in \mathbb{N}^{*}$.

\smallskip
{\it Step 2. Conclusion.} For $n\in \mathbb{N}^{*}$, we consider the following bootstrap assumption:
\begin{equation}\label{def:boott}
    t_{n}^{*}=\sup\left\{t\in [t_{n+1},t_{n}]:\lambda(t_{n+1})<\frac{3}{2}\lambda(t)\right\}>t_{n+1}.
\end{equation}
Under the time variable $s$, we denote 
\begin{equation*}
     s_{n}^{*}=s(t_{n}^{*})=-\int_{t_{n}^{*}}^{t_{1}}\frac{\dd \sigma}{\lambda^{3}(\sigma)}>s_{n+1}.
\end{equation*}
We claim that 
\begin{equation}\label{equ:tntn}
    t_{n}^{*}=t_{n}\quad 
    \mbox{and}\quad 
    s_{n}^{*}=s_{n},\quad \mbox{for any}\ n\in \mathbb{N}^{*}.
\end{equation}
Indeed, for any $t\in [t_{n+1},t_{n}^{*}]$, we can use the same argument as in {\it Step 1} to obtain 
    \begin{equation}\label{est:expt}
   \int_{\R^{2}}\varepsilon^{2}\left(t,y_{1},y_{2}\right)\dd y_{2}\le A_{4}e^{-\frac{y_{1}}{A_{4}}},\quad\mbox{for any}\ y_{1}>0.
\end{equation}
Here, $A_{4}>1$ is a universal constant independent of $\phi(t)$. 
Therefore, from the Cauchy-Schwarz inequality, we find
\begin{equation*}
\int_{\R}\int_{0}^{\infty}y_{1}^{100}\varepsilon^{2}(t,y)\dd y_{1}\dd y_{2}\lesssim 1,\quad \mbox{for any}\ t\in [t_{n+1}, t^{*}_{n}].
\end{equation*}
It follows from~Proposition~\ref{pro:decommini} and the Cauchy-Schwarz inequality that 
\begin{equation*}
    \int_{\R}\int_{0}^{\infty}y_{1}^{8}\varepsilon^{2}(t,y)\dd y_{1}\dd y_{2}
    \lesssim  
    \left(\int_{\R}\int_{0}^{\infty}y_{1}^{100}\varepsilon^{2}(t,y)\dd y_{1}\dd y_{2}\right)^{\frac{8}{100}}\|\varepsilon(t)\|_{L^{2}}^{\frac{46}{25}}\lesssim \lambda^{\frac{46}{25}}(t).
\end{equation*}
Combining the above estimate with Proposition~\ref{pro:decommini}, we see that (H1)--(H3) hold on $[t_{n+1},t^{*}_{n}]$ and thus these assumptions hold on $[s_{n+1},s_{n}^{*}]$ under the variable $s$. 
Actually, we could take $t_{1}$ very close to $T$ and thus $\phi(t)$ satisfies (H1)--(H3).
The solution is therefore in the monotonicity regime of Lemma \ref{le:refin} and Proposition~\ref{prop:energymoni}.
Denote $\widetilde{\lambda}(s)=\lambda(s)(1-J_{1}(s))^{2}$. Then, from Lemma \ref{le:refin} and Proposition~\ref{prop:energymoni}, 
\begin{equation}\label{est:lambdati}
    \bigg|\frac{\widetilde{\lambda}_{s}}{\widetilde{\lambda}}+b\bigg|
    \lesssim \mathcal{N}_{0}+b^{2}\Longrightarrow
    \frac{\widetilde{\lambda}_{s}}{\widetilde{\lambda}}\gtrsim -\mathcal{N}_{0}.
\end{equation}
Next,  from (ii) of Lemma~\ref{le:refin}, we deduce that 
\begin{equation*}
    b_{s}+\theta b^{2}=O\left(\mathcal{N}_{0}+|b|\mathcal{N}_{0}^{\frac{1}{2}}+|b|^{3}\right)\Longrightarrow
\frac{1}{2}b^{2}+\frac{b_{s}}{\theta}\lesssim \mathcal{N}_{0}.
\end{equation*}
Based on the above identity and $b_{s}b^{2}=\frac{1}{3}\frac{\dd}{\dd s}\left(b^{3}\right)$, we find
\begin{equation*}
    \int_{s_{n+1}}^{s^{*}_{n}}b^{4}(\sigma)\dd \sigma
    \lesssim 
     |{b(s_{n}^{*})}|^{3}
    +|{b(s_{n+1})}|^{3}
    +\kappa^{2}\int_{s_{n+1}}^{s^{*}_{n}}\mathcal{N}_{0}(\sigma)\dd \sigma.
    \end{equation*}
    On the other hand, from Proposition~\ref{prop:energymoni}, we also find
    \begin{equation*}
        {\mathcal{N}_{1}(s^{*}_{n})}+
        \int_{s_{n+1}}^{s^{*}_{n}}{\mathcal{N}_{0}(\sigma)}\dd \sigma
        \lesssim \mathcal{N}_{1}(s_{n+1})+\int_{s_{n+1}}^{s^{*}_{n}}b^{4}(\sigma)\dd \sigma.
    \end{equation*}
    Combining the above two estimates, we see that
    \begin{equation*}
        \int_{s_{n+1}}^{s^{*}_{n}}b^{4}(\sigma)\dd \sigma
        +\int_{s_{n+1}}^{s^{*}_{n}}\mathcal{N}_{0}(\sigma)\dd \sigma\lesssim
        \mathcal{N}_{1}(s_{n+1})
    +|b(s_{n}^{*})|^{3}+|b(s_{n+1})|^{3}.
    \end{equation*}
   Putting the above estimate together with~\eqref{est:lambdati}, for any $s\in [s_{n+1},s_{n}^{*}]$, we conclude that
    \begin{equation*}
        \log \left(\frac{\widetilde{\lambda}(s)}{\widetilde{\lambda}(s_{n+1})}\right)\gtrsim -\kappa
        \Longrightarrow \frac{\widetilde{\lambda}(s)}{\widetilde{\lambda}(s_{n+1})}\ge e^{O(\kappa)}\Longrightarrow 
        \frac{{\lambda}(s)}{{\lambda}(s_{n+1})}\ge 1+O(\kappa).
    \end{equation*}
 This strictly improves the bootstrap assumption in~\eqref{def:boott} by taking $\kappa$ sufficiently small, thereby completing the proof of~\eqref{equ:tntn}. From the definition of $t_{n}^{*}$ and~\eqref{equ:tntn}, we directly have 
    \begin{equation*}
      0<  \lambda(\bar{t})<\frac{3}{2}\lambda({t}),\quad \mbox{for any}\ T<\bar{t}<{t}<t_{1}
    \end{equation*}
Finally,  using the argument in {\it Step 1} via Lemma~\ref{le:monomass}  again, we complete the proof of Proposition~\ref{prop:decay}.
    \end{proof}

\subsection{End of the proof of Theorem~\ref{thm:main}}

We are in a position to complete the proof of Theorem~\ref{thm:main}. In the remainder of the proof, the implied constants in $\lesssim$ and $O$ can depend on the large constant $B$.

\begin{proof}[End of the proof of Theorem~\ref{thm:main}]
The proof is based on a contradiction argument. Indeed, for the sake of contradiction, we assume that there exists a minimal mass blow-up solution $\phi(t)$ in the sense of Definition~\ref{def:minimalmass}. Then, a contradiction will follow from the following observations.

\smallskip
{\it Step 1. Entering the monotonicity regime.} Note that, from Propositon~\ref{prop:decay} and the Cauchy-Schwarz inequality, we find
\begin{equation*}
\sup_{T<t\le t_{1}}\int_{\R}\int_{0}^{\infty}y_{1}^{100}\varepsilon^{2}(t,y)\dd y_{1}\dd y_{2}\lesssim 1.
\end{equation*}
It follows from~Proposition~\ref{pro:decommini} and Cauchy-Schwarz inequality that 
\begin{equation*}
    \int_{\R}\int_{0}^{\infty}y_{1}^{8}\varepsilon^{2}(t,y)\dd y_{1}\dd y_{2}
    \lesssim  
    \left(\int_{\R}\int_{0}^{\infty}y_{1}^{100}\varepsilon^{2}(t,y)\dd y_{1}\dd y_{2}\right)^{\frac{8}{100}}\|\varepsilon(t)\|_{L^{2}}^{\frac{46}{25}}\lesssim \lambda^{\frac{46}{25}}(t).
\end{equation*}
Combining the above estimate with Proposition~\ref{pro:decommini}, we see that (H1)--(H3) hold on $(T,t_{1}]$ and thus these assumptions hold on $(-\infty,0]$ under the variable $s$. The solution is therefore in the monotonicity regime of Lemma~\ref{le:refin} and Proposition~\ref{prop:energymoni}.

\smallskip
{\it Step 2. Control of $\frac{b}{\lambda^{\theta}}$.} We claim that, for any 
$-\infty<s_{*}<s_{**}<0$,
\begin{equation}\label{est:blambda}
    \left|\frac{b(s_{*})}{\lambda^{\theta}(s_{*})}e^{J({s_{*}})}-
    \frac{b(s_{**})}{\lambda^{\theta}(s_{**})}e^{J(s_{**})}\right|\lesssim \frac{\mathcal{N}_{1}(s_{*})}{\lambda^{\theta}(s_{*})}+
    \frac{b^{2}(s_{*})}{\lambda^{\theta}(s_{*})}+\frac{b^{2}(s_{**})}{\lambda^{\theta}(s_{**})}.
\end{equation}
Indeed, from (ii) of Lemma~\ref{le:refin}, we deduce that 
\begin{equation*}
    b_{s}+\theta b^{2}=O\left(\mathcal{N}_{0}+|b|\mathcal{N}_{0}^{\frac{1}{2}}+|b|^{3}\right)\Longrightarrow
\frac{1}{2}b^{2}+\frac{b_{s}}{\theta}\lesssim \mathcal{N}_{0}.
\end{equation*}
Based on the above identity and $b_{s}|b|=\frac{1}{2}\frac{\dd}{\dd s}\left(|b|b\right)$, we find
\begin{equation*}
    \int_{s_{*}}^{s_{**}}\frac{|b(\sigma)|^{3}}{\lambda^{\theta}(\sigma)}\dd \sigma
    \lesssim 
    \frac{b^{2}(s_{*})}{\lambda^{\theta}(s_{*})}+\frac{b^{2}(s_{**})}{\lambda^{\theta}(s_{**})}
    +\int_{s_{*}}^{s_{**}}\frac{\mathcal{N}_{0}(\sigma)}{\lambda^{\theta}(\sigma)}\dd \sigma.
    \end{equation*}
    On the other hand, from Proposition~\ref{prop:energymoni}, we also find
    \begin{equation*}
        \frac{\mathcal{N}_{1}(s_{**})}{\lambda^{\theta}(s_{*})}+
        \int_{s_{*}}^{s_{**}}\frac{\mathcal{N}_{0}(\sigma)}{\lambda^{\theta}(\sigma)}\dd \sigma\lesssim \frac{\mathcal{N}_{1}(s_{*})}{\lambda^{\theta}(s_{*})}+\int_{s_{*}}^{s_{**}}\frac{b^{4}(\sigma)}{\lambda^{\theta}(\sigma)}\dd \sigma.
    \end{equation*}
    Combining the above two estimates, we see that
    \begin{equation*}
        \int_{s_{*}}^{s_{**}}\frac{|b(\sigma)|^{3}}{\lambda^{\theta}(\sigma)}\dd \sigma
        +\int_{s_{*}}^{s_{**}}\frac{\mathcal{N}_{0}(\sigma)}{\lambda^{\theta}(\sigma)}\dd \sigma\lesssim
        \frac{\mathcal{N}_{1}(s_{*})}{\lambda^{\theta}(s_{*})}
    +\frac{b^{2}(s_{*})}{\lambda^{\theta}(s_{*})}+\frac{b^{2}(s_{**})}{\lambda^{\theta}(s_{**})}.
    \end{equation*}
    Next, from the definition of $J(s)$, we obtain 
    \begin{equation*}
        \left|e^{J(s)}-1\right|\lesssim |J(s)|\lesssim \mathcal{N}_{0}^{\frac{1}{2}}(s)\lesssim \alpha(\kappa)\ll 1,\quad \mbox{for any}\ s\in (-\infty,0).
    \end{equation*}
    Therefore, from the above estimate and (iv) of Lemma~\ref{le:refin}, we have 
\begin{equation*}
\begin{aligned}
&\left|
\frac{b(s_{*})}{\lambda^{\theta}(s_{*})}e^{J(s_{*})}-\frac{b(s_{**})}{\lambda^{\theta}(s_{**})}e^{J(s_{**})}
\right|\\
&\lesssim \int_{s_{*}}^{s_{**}}\left|\frac{\dd}{\dd \sigma}\left(\frac{b}{\lambda^{\theta}}e^J\right)(\sigma)\dd \sigma\right|
\lesssim \frac{\mathcal{N}_1(s_{*})}{\lambda^{\theta}(s_{*})}
+\frac{b^2(s_{*})}{\lambda^{\theta}(s_{*})}+\frac{b^2(s_{**})}{\lambda^{\theta}(s_{**})},
\end{aligned}
\end{equation*}
which means that the estimate~\eqref{est:blambda} is true.

\smallskip
{\it Step 3. Conclusion.} First, from~\eqref{est:blambda}, for any $-\infty<s_{*}<s_{**}<0$  with $-s_{**}$ large enough, we deduce that 
\begin{equation*}
    \frac{|b(s_{**})|}{2\lambda^{\theta}(s_{**})}\le \frac{|b(s_{*})|}{\lambda^{\theta}(s_{*})}
    \le  2\frac{|b(s_{**})|}{\lambda^{\theta}(s_{**})}.
\end{equation*}
On the other hand, from Proposition~\ref{pro:decommini}, for any $s\in (-\infty,0)$, we deduce that
\begin{equation*}
    |b(s)|\lesssim \lambda^{2}(s)E_{0}\Longrightarrow
    \frac{|b(s)|}{\lambda^{\theta}(s)}\lesssim \lambda^{2-\theta}(s)E_{0}.
\end{equation*}
Combining the above two estimates with $\lambda(s)\to 0$ as $s\downarrow -\infty$, we reach a contradiction. At this point, we have proved the nonexistence of $\phi(t)$ and thus the proof of Theorem~\ref{thm:main} is complete.
\end{proof}

\end{document}